\documentclass[11pt]{amsart}
\usepackage[utf8]{inputenc}
\usepackage{amsmath}
\usepackage{hyperref}
\usepackage{amssymb}
\usepackage{euscript}
\usepackage{bbold}
\usepackage{amssymb}
\usepackage{mathrsfs}
\usepackage{graphicx}

\usepackage{mathtools, amssymb, amsfonts, amsthm, enumitem}
\usepackage{epsfig, psfrag, color, multicol, graphicx, tikz,pdfsync}
\usepackage{amsfonts}
\usepackage{epstopdf}
\usepackage{xspace}

\usepackage{hyperref}

\usepackage[all]{xy}

\usepackage{amsthm}
\usepackage{amsmath}

\newtheorem{definition}{Definition}
\newtheorem{hypothesis}{Conjecture}
\newtheorem{theorem}{Theorem}
\newtheorem{statement}{Statement}

\newtheorem{step}{Step}
\newtheorem{lemma}{Lemma}

\DeclareMathOperator{\Vol}{Vol}

\DeclareMathOperator{\diag}{diag}
\DeclareMathOperator{\const}{const}

\def \Rn {\mathbb{R}^n}
\def \R {\mathbb{R}}
\def \E {\mathcal{E}}

\def \v {\tilde V}

\title[Uniqueness of a 3D ellipsoid with given intrinsic volumes]{
Uniqueness of a three-dimensional ellipsoid \\
with given intrinsic volumes}

\author[Fedor Petrov\and Alexander~Tarasov ]{Fedor~Petrov$^{\dagger,\star}$, \ \,  and  \  Alexander~Tarasov$^{\dagger}$}

\thanks{\thinspace ${\hspace{-.45ex}}^\dagger$St. Petersburg
State University. 29b,~
the 14th line of the Vasilievsky Island.
St.\ Petersburg, Russia.}

\thanks{\thinspace ${\hspace{-.45ex}}^\star$
St. Petersburg Department 
of Steklov Mathematical Institute 
of Russian Academy of Sciences.
27, Fontanka. St. Petersburg, Russia}

\thanks{\thinspace e-mails: f.v.petrov@spbu.ru,
science.tarasov@gmail.com}

\thanks{\thinspace \
\today}

\begin{document}

\maketitle
\begin{abstract}
   Let $\mathcal{E}$ be an ellipsoid in $\mathbb{R}^n$. A. Gusakova and D. Zaporozhets conjectured that $\mathcal{E}$ is uniquely (up to rigid
   motions) determined by its intrinsic volumes. We prove this conjecture for $n = 3$.
\end{abstract}
\section{Introduction}

\subsection{Intrinsic volumes}
 
For a bounded convex set $K\subset \mathbb{R}^n$ the \textit{intrinsic volumes} $V_0(K),\dots, V_n(k)$ are defined as the coefficients in the Steiner formula 
\begin{equation}
    \Vol(K+t B_n)=\sum_{k=0}^n \kappa_{n-k}V_k(K)t^{n-k},
\end{equation}
where $B_n$ denotes the Euclidean unit ball in $ \mathbb{R}^n$, $\kappa_k=\pi^{k/2}/\Gamma(\frac{k}{2}+1)$ denotes the volume of $B_k$, and $\Vol$ denotes the $n-$dimensional volume. Kubota's  formula
states that 
\begin{equation}
    V_k(K)=\binom{n}{k}\frac{\kappa_{n}}{\kappa_{k}\kappa_{n-k}}\int_{G_n,m}\Vol_k(p_{\nu}(K))d\omega (\nu),\quad 1\le k \le n.
\end{equation}
Here $G_{n,k}$ denotes the Grassmannian of all $k$-dimensional linear subspaces of $\mathbb{R}^n$; $p_{\nu}(K)$ denotes the orthogonal projection of $K$ to $\nu \in G_{n,k}$; $d\omega$ is the $O(n)$-invariant probabilistic measure
on $G_{n,k}$ (see \cite{Burago}, 19.3.2).

In particular, if $k=1$, $n-1$, $n$, then $V_k(K)$ coincides up to a constant factor with the so-called mean width, surface area and
the $n$-dimensional volume of $K$.

It is clear that in general the 
convex body can not be determined by the sequence of its intrinsic volumes,
but this may be expected to be the case
for certain natural $n$-parametric families
of convex bodies. 
For rectangular parallelepipeds, 
their intrinsic volumes up to constant factor are elementary symmetric functions of the edge lengths. Hence by Vieta theorem the edge lenghts are the roots of the corresponding polynomial. Therefore we can uniquely recover the
edge lengths of the 
rectangular parallelepiped by its intrinsic volumes.

Anna Gusakova and Dmitry Zaporozhets (2017) conjectured the
uniqueness in the class of ellipsoids:

\begin{hypothesis}
If $\mathcal{E}_1,
\mathcal{E}_2$ are two ellipsoids in $\mathbb{R}^n$ such that
$V_i(\mathcal{E}_1)=V_i(\mathcal{E}_2)$
for all $i=1,2,\ldots,n$, then 
$\mathcal{E}_1$ and $\mathcal{E}_2$
are congruent.
\end{hypothesis}

For $n=1$ and $n=2$ this conjecture is 
quite simple. The main goal of this note is to prove it in dimension $3$, when it can be formulated as follows.

\begin{theorem}
If the volume, surface area and mean width of two ellipsoids in $ \mathbb{R}^3 $ are the same, then the ellipsoids are congruent.
\end{theorem}

\subsection{Explicit formulas for intrinsic volumes of ellipsoids}
Using Tsirelson's formula (see  \cite{Tsirelson} or Theorem 1.9 in
\cite{Ellipsoids}
for details) one can obtain the following expression for intrinsic volumes of ellipsoid
$\mathcal{E}$
with semiaxes $\{a_i\}_{i=1}^n$:

\begin{equation}\label{FurmulaIntrinsicVolumesEllipsoids}
        V_m(\mathcal{E})=\frac{(2\pi)^{m/2}}{m!}\mathbb{E}\sqrt{\det \left(MM^{\top}\right)},
\end{equation}
where the random rows
$\xi_1,\dots, \xi_m\in \mathbb{R}^m$ are i.i.d. $\sim \mathcal{N}(0,\diag(a_1^2,\dots, a_n^2))$ and $M$ 
is the $m\times n$ matrix whose rows are $\xi_1,\dots, \xi_m$. In other words, $\sqrt{\det \left(MM^{\top}\right)}$ is the $m$-dimensional
volume of the parallelepiped with the edge
vectors $\xi_1,\ldots,\xi_m$.

Taking $n = 3$, $m = 1$ in  $\eqref{FurmulaIntrinsicVolumesEllipsoids}$  we obtain the expression for the mean width of $\mathcal{E}$:
\begin{equation} \label{width}
    V_1(\mathcal{E})=\sqrt{2\pi}\mathbb{E}\sqrt{\langle\xi_1, \xi_1\rangle}=\sqrt{2\pi}\mathbb{E}\sqrt{a_1^2x^2+a_2^2y^2+a_3^2z^2}, \ \ x,y,z \sim \mathcal{N}(0,1).
\end{equation}

The next relation (see \cite{Ellipsoids}, prop. 4.8) states a duality between $V_k$ and $V_{n-k}$. Consider the following ellipsoids in $\mathbb{R}^n$:
\begin{equation*}
    \mathcal{E}=\{ x\in \Rn: \ \sum_{i=1}^{n}a_i^2x_i^2 \le 1\}, \ \ \mathcal{E}^*=\{ x\in \Rn: \ \sum_{i=1}^{n}\frac{x_i^2}{a_i^2} \le 1\}.
\end{equation*}
Then
\begin{equation*}
    V_k(\mathcal{E})=  \frac{\kappa_k}{\kappa_n \kappa_{n-k}}V_n(\E)V_{n-k}(\E^*).
\end{equation*}
Again, taking $n=3$ and $k=2$, we obtain
\begin{equation} 
    V_2(\E)=\frac{\pi}{2}a_1a_2a_3V_1(\E^*)=\frac{\pi^{3/2}}{\sqrt{2}} a_1a_2a_3\mathbb{E}\sqrt{\frac{1}{a_1^2}x^2+\frac{1}{a_2^2}y^2+\frac{1}{a_3^2}z^2}, \ \ x,y,z \sim \mathcal{N}(0,1).
\end{equation}

\section{Proof of Theorem 1}

Now we have the explicit formulas for intrinsic volumes of ellipsoids in $\mathbb{R}^3$. From now on we will use the notation $(a,b,c)$ for semiaxes of ellipsoids instead of $(a_1,a_2,a_3)$.

We parametrize the family of three-dimensional ellipsoids by semiaxes and so we identify it with $\mathbb{R}_+^3$.

\begin{definition} The function V is given by
\begin{equation}
     V(a,b,c)=\biggl(V_1(a,b,c),V_2(a,b,c),V_3(a,b,c)\biggr),
\end{equation}
where $V_j(a,b,c)$, $j=1$, $2$, $3$, is 
the $j$-th intrinsic volume of the 
ellipsoid with semiaxes $a,b,c$.
\end{definition}

This parameterization of ellipsoids by
semiaxes is not bijective on $\mathbb{R}_+^3$, but becomes bijective if
we restrict it to the
set of parameters $\{(a,b,c)\in \R_+^3:a\ge b \ge c\}.$

Fix a point $(a_0,b_0,c_0) \in \mathbb{R}_+^3$ such that
\begin{equation} \label{propertyoffixedpoint}
    a_0>b_0 \ge c_0 \ \ or \ \ a_0\ge b_0 > c_0
\end{equation}
For $i\in \{1, \ 3\}$ let $M_{V_i}=M_{V_i}(a_0,b_0,c_0)$ denote the level set of $V_i$:
\begin{equation*}
    M_{V_i}=\{(a,b,c) \in \R_+^3 : \ V_i(a,b,c)=V_i(a_0,b_0,c_0)\}.
\end{equation*}
Then $M_{V_3}$ is a smooth 2-dimensional manifold given by the equation $abc=\const$, 
and $M_{V_1}$ is a smooth 2-dimensional manifold since $V_1$ is a smooth function with nonzero gradient (this is clear
from differentiating \eqref{width},
see also the computations below.)
\begin{step}
    The manifolds $M_{V_1}$ and $M_{V_3}$ intersect transversally.
\end{step}

\begin{proof}
The reason is that 
the gradient vectors $\nabla V_1$ and $\nabla V_3$ 
have the opposite orders of coordinates: \begin{equation} \label{partial}\frac{\partial V_1}{\partial a}\ge\frac{\partial V_1}{\partial b}\ge\frac{\partial V_1}{\partial c}\ \ \text{and} \ \  \frac{\partial V_3}{\partial a}\le\frac{\partial V_3}{\partial b}\le\frac{\partial V_3}{\partial c},
\end{equation}
and the inequalities are strict
when corresponding semiaxes are not
equal. We have

$$\nabla V_3(a,b,c)=\nabla\left( \frac{4\pi}{3} abc\right)=\frac{4\pi}{3}abc\cdot\left(\frac{1}{a},\frac{1}{b},\frac{1}{c}\right),$$ so the inequality for partial derivatives of $V_3$ is clear.

We differentiate \eqref{width} and obtain
\begin{multline*}
    \frac{1}{\sqrt{2\pi}}\frac{\partial V_1}{\partial a}(a,b,c)=\mathbb{E}\frac{ax^2}{\sqrt{a^2{x^2}+b^2{y^2}+c^2{z^2}}}=\mathbb{E}\frac{x^2}{\sqrt{{x^2}+\frac{b^2}{a^2}{y^2}+\frac{c^2}{a^2}{z^2}}}\ge \\
    \mathbb{E}\frac{x^2}
    {\sqrt{x^2+\frac{a^2}{b^2}{y^2}+\frac{c^2}{b^2}{z^2}}}=\frac{1}{\sqrt{2\pi}}\frac{\partial V_1}{\partial b}(a,b,c),
\end{multline*}
analogously 
$\frac{\partial V_1}{\partial b}(a,b,c)\ge
\frac{\partial V_1}{\partial c}(a,b,c)$,
and equalities hold only if $a=b$ or
$b=c$, respectively.

Therefore the vectors $\nabla V_1$ and $\nabla V_3$ are collinear 
if and only if $a=b=c$. 
But in this case $V_1=(\frac{48}{\pi} V_3)^{1/3}$ is minimal
possible for fixed $V_3$
(see, for example, 
\cite{Burago}, section 20.2, or
apply the isoperimetric inequality
for the ellipsoid with semi-axes
$1/a,1/b,1/c$), 
and the equality is only possible for
a ball. Hence by $\eqref{propertyoffixedpoint}$ $M_{V_1}\cap M_{V_3}$ 
do not contain any point 
with three equal coordinates. 
Therefore the manifolds $M_{V_1}$ and $M_{V_3}$ intersect transversally.
\end{proof}

Further we need also the following

\begin{lemma}\label{2}
The intersection $M_{V_1}\cap M_{V_3}$
contains the unique point of the form
$(a,a,b),a<b$ and the unique point
of the form
$(c,c,d),d>c$.
\end{lemma}

\begin{proof}
Consider the curve
$\gamma(t)=(t,t,C/t^2)\subset M_{V_3}$,
where $C=V_3(a_0,b_0,c_0)$.
For $t=C^{1/3}$ the function
$V_1$ takes its minimal value,
and this minimum is strictly less than
$V_1(a_0,b_0,c_0)$.
for large or small $t>0$ it tends to
infinity (since the mean width is an inclusion-monotone 
function of the convex body, 
and the mean width of the long segment 
is large.) Thus by continuity
it suffices to prove
that if the derivative of $V_1$
along $\gamma$ equals to zero
at point $t_0$, then
$t_0=C^{1/3}$.
Note that the the gradients
of both functions $V_3$ and $V_1$
have the form $(A,A,B)$
at points of $\gamma$. The
gradient of $V_3$ is orthogonal to the
tangent vector $(1,1,-2C/t_0^3)$ of the
curve $\gamma(t)$ at $t_0$, since $V_3$
is constant along $\gamma$. This
orthogonality rewrites as 
$A=C\cdot B/t_0^3$. If the
gradient of $V_1$ is also orthogonal
to $(1,1,-2C/t_0^3)$, then these
two gradients are proportional. 
But we have already proved
that this holds only if
$t_0=C^{1/3}$.
\end{proof}

\begin{step}
The set $N:=M_{V_1}\cap  M_{V_3}$
is diffeomorphic to a union of several circles.
\end{step}

\begin{proof}
The set $M_{V_1}$ is bounded 
by the aforementioned
monotonocity argument. 
By the implicit function theorem, the intersection of two smooth 2-dimensional transversally intersecting manifolds in $\R^3$ is 1-dimensional smooth manifold. Since $M_{V_1}$ is bounded, $M_{V_1}\cap  M_{V_3}$ is a 
compact 1-dimensional smooth manifold. 
Hence it is diffeomorphic to a union of several circles.
\end{proof}

\begin{step}
There does not exist another point
$(a_1,b_1,c_1), a_1\geqslant b_1\geqslant c_1$,
with the same values of $V_1,V_2,V_3$ as at
the point
$(a_0,b_0,c_0)$.
\end{step}

\begin{proof}
Now we formulate the 
crucial lemma, whose proof is 
postponed to Section 3.
\begin{lemma}\label{1}
Jacobian of $V(a,b,c)$ is non-zero
 on the set $\{(a,b,c)\in \R^3: \ a>b>c>0\}$.
\end{lemma}

By the previous Step, 
the set $N:=M_{V_1}\cap  M_{V_3}$ is diffeomorphic to the union of several circles,
and 
$N$ contains exactly 
6 points
with two equal coordinates by Lemma \ref{2}.

On the other hand,
any connected component $\gamma$ of 
$N$ must contain at least
two points with equal coordinates.
Indeed, by Lemma \ref{1} any 
point $p$ on 
$\gamma$ with locally maximal
or locally minimal
value of $V_2$
must have
two equal coordinates
(since the derivatives of all
three functions $V_1,V_2,V_3$ 
along $\gamma$ are equal to 0.)

Denote by $S_{ij}$ the transposition of
the $i$-th and $j$-th coordinates, say
$S_{12}((x,y,z))=(y,x,z)$. Obviously
set 
$N$ is invariant under 
all these symmetries.

Let $\gamma$ be a connected component
of $N$ which contains
the point $p_0=(a,a,b)$, $a<b$. Then $S_{12}(\gamma)$
is also a connected component of $N$
containing
$p_0$. So $S_{12}(\gamma)=\gamma$. 
Analogously, if $\gamma$ contains
a point with another pair of equal
coordinates, it gives another
symmetry, $S_{13}$ or $S_{23}$, which
preserves $\gamma$, and $\gamma$ is invariant
under all the symmetries. In this case
all 6 points from $N$
with two equal coordinates belong to $\gamma$,
and $N=\gamma$.

If not, the second
point $q_0\in\gamma\setminus \{p_0\}$ with two equal coordinates should be
$q_0=(c,c,d)$, $c>d$. But then by continuity
there exists a point on $\gamma$ 
between $p_0$ and $q_0$ with equal
second and third coordinates. The
contradiction. 

Therefore $N$ is a single circle, and
six points on $N$ have equal coordinates.
The intervals between these six points
belong to six Weyl chambers
(corresponding to the six orderings
of coordinates).
Consider two our points $(a_0,b_0,c_0)$ and
$(a_1,b_1,c_1)$ in $M$ which
belong to the same
closed Weyl chamber
$\{a\geqslant b\geqslant c\}$. 
If the function
$V_2$ takes the same value
at these two points, it has a local
maximum or minimum strictly
between them. But 
such a point should have two equal coordinates
as noted before. The contradiction.
\end{proof}

\section{Proof of Lemma \ref{1}}
\subsection{Explicit formula for the Jacobian matrix}
It will be more convenient for us to consider functions $\tilde V_1,\tilde V_2$ and $\tilde V_3$ given by
\begin{equation*}
    \tilde V_3(a,b,c)=\frac{3}{4\pi}V_3(e^a,e^b,e^c)=e^a e^b e^c=e^{a+b+c},
\end{equation*}
\begin{equation*}
\tilde V_2(a,b,c)=\frac{4\sqrt{2}}{3\sqrt{\pi}}\cdot \frac{V_2(e^a,e^b,e^c)}{V_3(e^a,e^b,e^c)}=\mathbb{E}\sqrt{e^{-2a}{x^2}+e^{-2b}{y^2}+e^{-2c}{z^2}},  \ \ x,y,z \sim \mathcal{N}(0,1),
\end{equation*}
\begin{equation*}
\tilde V_1(a,b,c)=\frac{1}{\sqrt{2\pi}}V_1(e^a,e^b,e^c)=\mathbb{E}\sqrt{e^{2a}{x^2}+e^{2b}{y^2}+e^{2c}{z^2}},  \ \ x,y,z \sim \mathcal{N}(0,1).
\end{equation*}
We first compute gradients of these functions:
\begin{equation}
    \nabla \v_3 (a,b,c) =\nabla e^{a+b+c}=e^{a+b+c} \cdot (1,1,1),
\end{equation}

\begin{equation*}
    \v_2 (a,b,c)_a'=\left(\mathbb{E}\sqrt{e^{-2a}{x^2}+e^{-2b}{y^2}+e^{-2c}{z^2}}\right)_a'=-\mathbb{E}\frac{e^{-2a}x^2}{\sqrt{e^{-2a}{x^2}+e^{-2b}{y^2}+e^{-2c}{z^2}}},
\end{equation*}
hence
\begin{multline*}
\nabla \v_2(a,b,c)=\\
=-\left(\mathbb{E}\frac{e^{-2a}x^2}{\sqrt{e^{-2a}{x^2}+e^{-2b}{y^2}+e^{-2c}{z^2}}},\mathbb{E}\frac{e^{-2b}y^2}{\sqrt{e^{-2a}{x^2}+e^{-2b}{y^2}+e^{-2c}{z^2}}},\mathbb{E}\frac{e^{-2c}z^2}{\sqrt{e^{-2C}{x^2}+e^{-2b}{y^2}+e^{-2c}{z^2}}}\right).\end{multline*}
The same way we obtain:
\begin{equation*}
    \nabla \v_1(a,b,c)=\left(\mathbb{E}\frac{e^{2a}x^2}{\sqrt{e^{2a}{x^2}+e^{2b}{y^2}+e^{2c}{z^2}}},
    \mathbb{E}\frac{e^{2b}y^2}{\sqrt{e^{2a}{x^2}+e^{2b}{y^2}+e^{2c}{z^2}}},
    \mathbb{E}\frac{e^{2c}z^2}{\sqrt{e^{2a}{x^2}+e^{2b}{y^2}+e^{2c}{z^2}}}\right).
\end{equation*}
Now we define auxiliary function, in terms of which the Jacobi matrix can be conveniently written as follows

\begin{definition}
$$G(a,b,c)=\mathbb{E}\frac{a^2 x^2}
    {\sqrt{a^2x^2+b^2y^2+c^2z^2}},\ \text{where}\  \  x,y,z \sim \mathcal{N}(0,1). $$
\end{definition}

\begin{statement}
In this notation Jacobi matrix of $V$ has the following form:

$$J_{\v}(a,b,c)=
\begin{pmatrix}
1&1&1 \\
G(e^{-a},e^{-b},e^{-c})&
G(e^{-b},e^{-a},e^{-c})&
G(e^{-c},e^{-a},e^{-b})\\
G(e^{a},e^{b},e^{c})&
G(e^{b},e^{a},e^{c})&
G(e^{c},e^{a},e^{b})
\end{pmatrix}
$$

\end{statement}
It is sufficient to prove that the following matrix is nondegenerate for $a>b>c>0$:
\begin{equation}\label{matrix}
\begin{pmatrix}
1 &
1&
1 \\
G(\frac{1}{a},\frac{1}{b},\frac{1}{c})&
G(\frac{1}{b},\frac{1}{a},\frac{1}{c})&
G(\frac{1}{c},\frac{1}{a},\frac{1}{b})\\
G(a,b,c)&
G(b,a,c)&
G(c,a,b)
\end{pmatrix}.\end{equation}

\subsection
{Alternative formula for the function 
\emph{G(a,b,c)}}
Using the Gaussian integral
\begin{equation*}
    \int_{\mathbb{R}} e^{-Ts^2}ds=\frac{\sqrt{\pi}}{\sqrt{T}}
\end{equation*}
with  $T = a^2x^2+b^2y^2+c^2z^2$
we rewrite the formula of function $G$:
\begin{multline*}
    G(a,b,c)=\frac{1}{(2\pi)^{3/2}} \cdot
    \iiint_{\mathbb{R}^3} \frac{a^2x^2}{\sqrt{a^2x^2+b^2y^2+c^2z^2}} e^{-1/2(x^2+y^2+z^2)} dxdydz=\\
    =\frac{1}{2^{3/2}\pi^2} \cdot\int_{\mathbb{R}} \iiint_{\mathbb{R}^3} a^2x^2 e^{-1/2(x^2+y^2+z^2)-s^2(a^2x^2+b^2y^2+c^2z^2)} dxdydzds =\\
    =\frac{1}{2^{3/2}\pi^2} \cdot\int_{\mathbb{R}} \int_{\mathbb{R}} a^2x^2 e^{-x^2(1/2+s^2)}dx \int_{\mathbb{R}}  e^{-y^2(1/2+s^2)}dy\int_{\mathbb{R}}  e^{-z^2(1/2+s^2)}dzds =\\
    =\frac{1}{2\sqrt{2\pi}} \cdot \int_{\mathbb{R}} \frac{a^2}{(\sqrt{a^2s^2+1/2})^3}\frac{1}{\sqrt{b^2s^2+1/2}}\frac{1}{\sqrt{c^2s^2+1/2}}ds.
\end{multline*}
Using the above formulas and applying a 
change of variables $s \to s/\sqrt{2}$ we  obtain
\begin{equation*}
    G(a,b,c)=\sqrt{\frac{2}{\pi}} \cdot \int_{\mathbb{R}} \frac{1}{s^2+\frac{1}{a^2}}\frac{1}{\sqrt{(a^2s^2+1)(b^2s^2+1)(c^2s^2+1)}}ds.
\end{equation*}
In the same way we see that
\begin{equation*}
    G(\frac{1}{a},\frac{1}{b},\frac{1}{c})=\sqrt{\frac{2}{\pi}} \cdot \int_{\mathbb{R}} \frac{1}{t^2+a^2}\frac{1}{\sqrt{(\frac{t^2}{a^2}+1)(\frac{t^2}{b^2}+1)(\frac{t^2}{c^2}+1)}}dt.
\end{equation*}
Now we write the determinant of \eqref{matrix} in the following form
\begin{equation*}
\int_{\mathbb{R}} \int_{\mathbb{R}}
\begin{vmatrix}
1 \ \ \ \  \  \ 1 \ \ \ \ \ \ \ 1 \\
\frac{1}{t^2+a^2} \ \ \frac{1}{t^2+b^2} \ \ \frac{1}{t^2+c^2} \\
\frac{1}{s^2+\frac{1}{a^2}} \ \ \frac{1}{s^2+\frac{1}{b^2}} \ \ \frac{1}{s^2+\frac{1}{c^2}} \\
\end{vmatrix}
\frac{1}{\sqrt{(\frac{t^2}{a^2}+1)(\frac{t^2}{b^2}+1)(\frac{t^2}{c^2}+1)}}\frac{1}{\sqrt{(a^2s^2+1)(b^2s^2+1)(c^2s^2+1)}}dtds.
\end{equation*}

The following identity holds:
\begin{equation*}
    \begin{vmatrix}
1 \ \ \ \  \  \ 1 \ \ \ \ \ \ \ 1 \\
\frac{1}{t^2+a^2} \ \ \frac{1}{t^2+b^2} \ \ \frac{1}{t^2+c^2} \\
\frac{1}{s^2+\frac{1}{a^2}} \ \ \frac{1}{s^2+\frac{1}{b^2}} \ \ \frac{1}{s^2+\frac{1}{c^2}} \\
\end{vmatrix}
=\frac{(a^2-b^2)(a^2-c^2)(b^2-c^2)(s^2t^2-1)}{(a^2s^2+1)(b^2s^2+1)(c^2s^2+1)(\frac{t^2}{a^2}+1)(\frac{t^2}{b^2}+1)(\frac{t^2}{c^2}+1)}.
\end{equation*}
So the determinant of \eqref{matrix} equals to
\begin{equation} \label{formula1}
\int_{\mathbb{R}} \int_{\mathbb{R}} \frac{(a^2-b^2)(a^2-c^2)(b^2-c^2)(s^2t^2-1)}{[(a^2s^2+1)(b^2s^2+1)(c^2s^2+1)(\frac{t^2}{a^2}+1)(\frac{t^2}{b^2}+1)(\frac{t^2}{c^2}+1)]^\frac{3}{2}} dtds.
\end{equation}
Note that the integrand  is even with respect to $s$ and $t$. Denote by $I$ the same integral but over the set $\mathbb{R}_+ \times \mathbb{R}_+$.
Applying the change of variables $s \to \frac{1}{x}, \ t \to \frac{1}{y}$ we get

\begin{equation*}
I=\int_{\mathbb{R}_{+}} \int_{\mathbb{R}_{+}} \frac{(a^2-b^2)(a^2-c^2)(b^2-c^2)(\frac{1}{x^2y^2}-1)}{[(\frac{a^2}{x^2}+1)(\frac{b^2}{x^2}+1)(\frac{c^2}{x^2}+1)(\frac{1}{y^2a^2}+1)(\frac{1}{y^2b^2}+1)(\frac{1}{y^2c^2}+1)]^\frac{3}{2}} \frac{1}{x^2y^2}dxdy.
\end{equation*}

This integral is similar to $I$: 
the integrands differ by the factor of $(xy)^5$. Indeed,
\begin{equation*}
I=\int_{\mathbb{R}_{+}} \int_{\mathbb{R}_{+}} \frac{(a^2-b^2)(a^2-c^2)(b^2-c^2)(1-x^2y^2)}{[(\frac{1}{x^2}+\frac{1}{a^2})(\frac{1}{x^2}+\frac{1}{b^2})(\frac{1}{x^2}+\frac{1}{c^2})(\frac{1}{y^2}+a^2)(\frac{1}{y^2}+b^2)(\frac{1}{y^2}+c^2)]^\frac{3}{2}} \frac{1}{x^4y^4}dxdy=
\end{equation*}
\begin{equation*}
=\int_{\mathbb{R}_{+}} \int_{\mathbb{R}_{+}} \frac{(a^2-b^2)(a^2-c^2)(b^2-c^2)(1-x^2y^2)}{[(xy)^{-6}(1+\frac{x^2}{a^2})(1+\frac{x^2}{b^2})(1+\frac{x^2}{c^2})(1+y^2a^2)(1+y^2b^2)(1+y^2c^2)]^\frac{3}{2}} \frac{1}{x^4y^4}dxdy=
\end{equation*}

\begin{equation*}
=\int_{\mathbb{R}_{+}} \int_{\mathbb{R}_{+}} \frac{(a^2-b^2)(a^2-c^2)(b^2-c^2)(1-x^2y^2)(xy)^5}{[(1+\frac{x^2}{a^2})(1+\frac{x^2}{b^2})(1+\frac{x^2}{c^2})(1+y^2a^2)(1+y^2b^2)(1+y^2c^2)]^\frac{3}{2}}dxdy.
\end{equation*}
Redenoting the variables we write this as

\begin{equation}\label{formula2}
    I=\int_{\mathbb{R}_{+}} \int_{\mathbb{R}_{+}} \frac{(a^2-b^2)(a^2-c^2)(b^2-c^2)(1-t^2s^2)(ts)^5}{[(1+\frac{t^2}{a^2})(1+\frac{t^2}{b^2})(1+\frac{t^2}{c^2})(1+s^2a^2)(1+s^2b^2)(1+s^2c^2)]^\frac{3}{2}}dtds.
\end{equation}
Now we consider the sum of the two copies of $I$, using $\eqref{formula1}$ and $\eqref{formula2}$. The denominators of integrands
 coincide. Then separately write out 
 the numerators divided by $(a^2-b^2)(a^2-c^2)(b^2-c^2)$
\begin{equation*}
    (s^2t^2-1)+(1-s^2t^2)(st)^5=((st)^2-1)(1-(st)^5)
\end{equation*}
This observation shows that $I+I$ is an integral of a negative function. Hence $I$ is not equal to zero on the set $\{a>b>c>0\}$, 
 that finishes the proof of
 Lemma:

\begin{equation*}
    I+I=(a^2-b^2)(a^2-c^2)(b^2-c^2)\int_{\mathbb{R}_{+}} \int_{\mathbb{R}_{+}} \frac{((st)^2-1)(1-(st)^5)}{[(1+\frac{t^2}{a^2})(1+\frac{t^2}{b^2})(1+\frac{t^2}{c^2})(1+s^2a^2)(1+s^2b^2)(1+s^2c^2)]^\frac{3}{2}}dtds.
\end{equation*}

\section{Acknowledgements}

We are grateful to D.~Zaporozhets 
for introducing us to this problem,
the permanent attention to the
work and helpful
advices;
and to P.~Nikitin and G.~Monakov
for fruitful discussions.

\addcontentsline{toc}{section}{References}


\end{document}